\documentclass[oneside,10pt]{article}
\usepackage[b5paper]{geometry}	

\usepackage{amsfonts,amsmath,latexsym,amssymb}
\usepackage{theorem}
\usepackage{mathrsfs,upref}
\usepackage{mathptmx}		

\usepackage{hyperref}

\usepackage{oam}	
\theoremstyle{definition}
\newtheorem{lemma}{Lemma}[section]
\newtheorem{theorem}{Theorem}[section]
\newtheorem{proposition}{Proposition}[section]
\newtheorem{remark}{Remark}[section]

\usepackage{color}
\definecolor{Myred}{cmyk}{0.0,1.0,1.0,0.00}
\usepackage{enumerate}
\usepackage{colordvi}
\usepackage{graphicx}

\begin{document}

\title[Spectral estimates for Dirichlet Laplacian]{Spectral estimates for Dirichlet Laplacian on tubes with exploding twisting velocity}

\author{ Diana Barseghyan,  Andrii Khrabustovskyi}

\address{Diana Barseghyan\\ 
Department of Theoretical Physics, Nuclear Physics Institute of the Czech Academy of Sciences\\ 
25068 \v{R}e\v{z} near Prague, Czech Republic \&\\ 
Department of Mathematics, University of Ostrava\\
30. dubna 22, 70200 Ostrava, Czech Republic
\\ \email{diana.barseghyan@osu.cz}}

\address{Andrii Khrabustovskyi\\ 
Institute of Applied Mathematics, Graz University of Technology,\\Steyrergasse 30, 8010 Graz, Austria
\\\email{khrabustovskyi@math.tugraz.at}}

\CorrespondingAuthor{Diana Barseghyan}

\date{\today}

\keywords{Dirichlet Laplacian, twisted tube, discrete spectrum, eigenvalue bounds}

\subjclass{35P20, 35P15, 81Q10, 81Q37}

\begin{abstract}
We study the spectrum of the Dirichlet Laplacian on an unbounded twisted  tube with twisting velocity exploding to infinity. If the tube cross section does not intersect the axis of rotation, then its spectrum is purely discrete under some additional conditions on the twisting velocity (D.~Krej\v{c}i\v{r}\'{i}k, 2015). 
In the current work we prove a Berezin type upper bound  for the eigenvalue moments.\end{abstract}

\maketitle

\section{Introduction} 
\label{s:intro}

Advances in mesoscopic physics have given rise to study spectral properties of unbounded regions of tubular shape.  The Dirichlet Laplacian in such domains is a reasonable model for the Hamiltonian in quantum-waveguide nanostructures.  
One of the peculiarities of such domains is that they may possess geometrically-induced bound states, which was noticed first in the two-dimensional situation by P.~Exner and P.~\v{S}eba \cite{ES89}, and studied intensively since then,  see, e.g., the papers \cite{CDFK05,DE95,GJ92,RB95} and the recent monograph \cite{EK15}. 

In the above mentioned papers bound states are generated by a local \emph{bending} of a straight waveguide. In the present work we deal with another class of unbounded tubular domains -- the so-called \emph{twisted tubes}. 

\emph{Twisted tube} is a set which is obtained by translating and rotating a bounded open connected set $\omega \subset \mathbb{R}^2$ about a straight line in $\mathbb{R}^3$. More precisely, for a given $x_1\in\mathbb{R}$ and $x:=(x_2, x_3)\in\omega$ we define the mapping $\mathfrak{L}:\: \mathbb{R}\times\omega\to\mathbb{R}^3$ by
\begin{equation}\label{L} 
\mathfrak{L}(x_1, x)=(x_1, x_2\cos\theta(x_1)+x_3\sin\theta(x_1), x_3\cos\theta(x_1)-x_2\sin\theta(x_1))\,.
\end{equation}
Here $\theta: \mathbb{R} \to \mathbb{R}$ is the \emph{rotation angle}   which is assumed to be a sufficiently regular function.
Then the region $\Omega:=\mathfrak{L}(\mathbb{R} \times \omega)\subset\mathbb{R}^3$ is a twisted tube  unless the function $\theta$ is constant or $\omega$ is rotationally symmetric with respect to the origin in $\mathbb{R}^2$ (i.e., either a
disk or an annulus with the center at the origin). 

In what follows for $\Omega\subset \mathbb{R}^n$, $n\ge1$, we denote by $-\Delta_D^\Omega$  the Dirichlet Laplacian in $L^2(\Omega)$. If $\Omega$ is bounded,  the spectrum of $-\Delta_D^\Omega$ is purely discrete. However, for unbounded domains  the discreteness of the spectrum is no longer guaranteed. A necessary condition is the so called \emph{quasi-boundedness} of $\Omega$ (see, e.g., \cite{EE87}) which is satisfied, by definition, if $\lim_{ {x\in\Omega},\, {|x|\rightarrow\infty}} \mathrm{dist}(x,\partial\Omega)=0\,.$

It is easy to see that the twisted tube $\Omega$ is 
not a quasi-bounded domain if the cross-section $\omega$
contains the origin in $\mathbb{R}^2$. Consequently, in this case
the essential spectrum of $-\Delta_D^\Omega$ is non-empty.
For example, if $\dot{\theta}$ vanishes at infinity then \cite{K08}
$\sigma(-\Delta_D^\Omega)=[\lambda_1, \infty),$ 
where $\lambda_1$ is the first eigenvalue of $-\Delta_D^\omega$ in $L^2(\omega)$. Another interesting example was treated in \cite{EK05}:   $\dot{\theta}(x_1)$ is a constant (we denote it $\beta$). In this case  
$
\sigma(-\Delta_D^\Omega)=[\lambda_1(\beta), \infty),
$
where $\lambda_1(\beta)$ is  the spectral  threshold of the two-dimensional operator $-\Delta_D^\omega-\beta^2 \partial_\tau ^2$ in $L^2(\omega)$,   $\partial_\tau$ is the angular momentum operator  $\partial_\tau:=x_3 \partial_2-x_2 \partial _3.$
Of course, if $\omega$ contains the origin it does not mean
that the spectrum is purely essential: for instance, 
if we perturb locally a twisted tube with constant $\dot\theta(x_1)=\beta$, then eigenvalues may appear below $\lambda_1(\beta)$, see \cite{EK05} for more details.

\begin{figure}[t]
\begin{center}
\includegraphics[width=80mm]{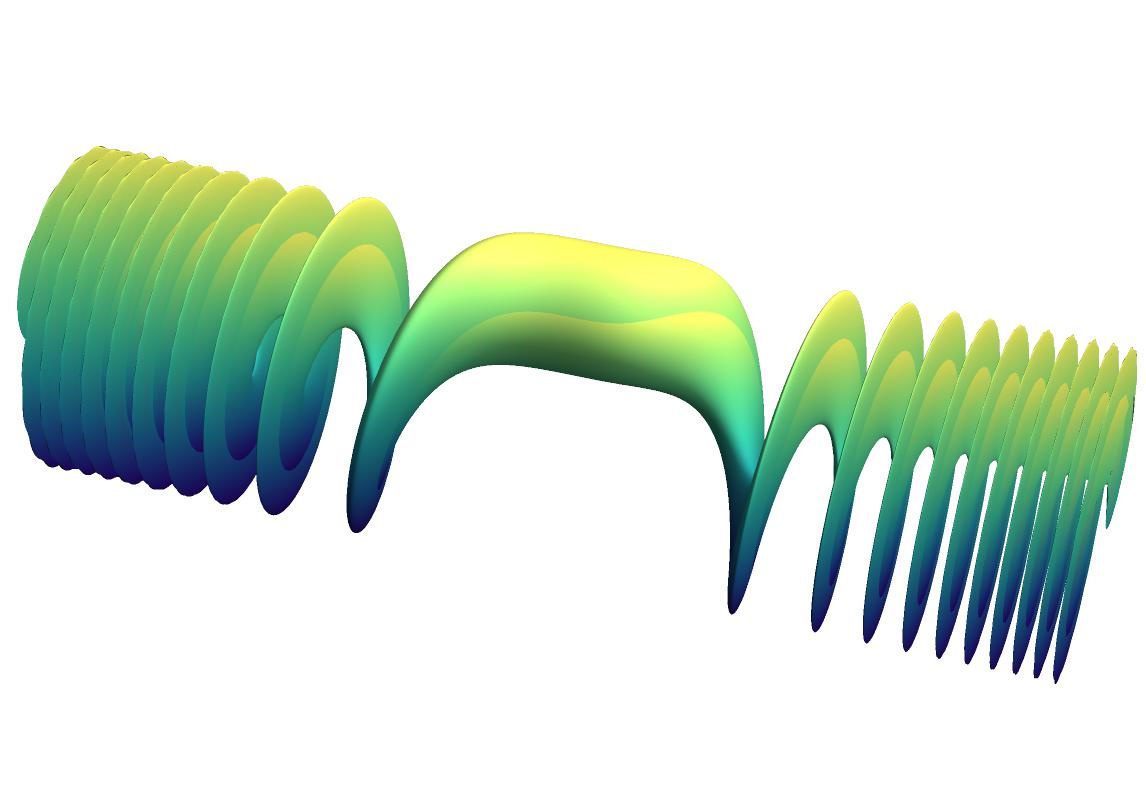}
\caption{The twisted tube $\Omega$}\label{figure1}
\end{center}
\end{figure}

The picture changes drastically if 
\begin{gather} 
\label{halfspace}
\omega\subset \{(x_2, x_3)\in \mathbb{R}^2 | \,x_2 >0\}.
\end{gather}
The corresponding twisted tube is depicted on Figure~\ref{figure1}.
In this case it turns out that  $\sigma(-\Delta_D^\Omega)$ is  purely discrete provided 
\begin{gather} \label{3condition}
\mathrm{lim}_{x_1\to\pm\infty} |\dot{\theta}(x_1)| = \infty,
\end{gather} and therefore $\Omega$ becomes quasi-bounded  \cite{K15}.

In the present note we study  some properties of
the discrete eigenvalues in the model considered in \cite{K15}. Our main result is the Berezin type bound for 
eigenvalue moments of order $\sigma\geq 0$. 

Recall, that the classical Berezin bound is the  estimate from above for the moments of eigenvalues of the Dirichlet Laplacian $-\Delta_D^\Omega $ on a \emph{bounded} domain $\Omega$ lying below a fixed $\Lambda>0$ \cite{B67}: 
\begin{equation}\label{classical-berezin}\mathrm{tr}(-\Delta_D^\Omega-\Lambda)_-^\sigma:=\sum_k(\lambda_k-\Lambda)_-^\sigma\le L_{\sigma, d}^{\mathrm{cl}} |\Omega| \Lambda^{\sigma+d/2},\,\,\sigma\ge 1.
\end{equation}
Here  $\{\lambda_k\}_{k\in\mathbb{N}}$ is 
 a sequence of eigenvalues of  $-\Delta_D^\Omega$, numbered in the ascending order with account of their multiplicities, $|\Omega|$ stands for the measure of $\Omega$, $L_{\sigma, d}^{\mathrm{cl}}$ is the so-called semiclassical constant given by 
\begin{equation}\label{semiclassical constant}L_{\sigma, d}^{\mathrm{cl}}=\frac{\Gamma(\sigma+1)}{(4\pi)^{d/2} \Gamma(\sigma+1+d/2)},
\end{equation}
and, finally, $(\cdot)_-$ is the negative part of the enclosed quantity (cf.~\eqref{pm}).
A similar inequality  holds also for $0\le\sigma<1$
with some, probably non-sharp, constant instead of  $L_{\sigma, d}^{\mathrm{cl}}$ \cite{L97}.

Unfortunately, for tubular domains we consider in the current work 
these estimates are meaningless since
their right-hand sides become infinite due to $|\Omega|=\infty$.
Nevertheless, we are able to derive a Berezin type bound
for twisted tubes whose rotation velocity explodes at infinity (see \eqref{3condition})
 and  additional technical assumptions \eqref{1condition}-\eqref{2condition} hold (see also Subsection~\ref{other}). 
The role of $|\Omega|\Lambda^{\sigma+3/2}$ will be played by $|\omega|$ times   certain expression involving $\Lambda$, $\theta(x_1)$ and $\dot\theta(x_1)$.  
 
Eigenvalue bounds for twisted tubes were also treated by P.~Exner and the first author in \cite{EB14}, where a locally  perturbed  twisted tube with constant rotation velocity  was considered.  
The authors derived Lieb-Thirring-type inequalities for eigenvalue moments of order $\sigma>1/2$. 
Other spectral aspects of  twisted tubes  were
treated in \cite{Gr04,Gr05} (existence/non-existence of bound states),   
\cite{BHK15,EKK08} (Hardy type inequalities), \cite{BMT07} (asymptotic behavior of the spectrum  as the thickness of the tube cross section goes to zero), \cite{BKRS09} (eigenvalue asymptotics in the case when the rotation velocity decays slowly at infinity).

The paper is organized as follows. In Section~\ref{sec-main} we 
present our main result (Theorem~\ref{th1}). Its proof  is  given in Section~\ref{sec-proof}. Finally, in Section~\ref{discussion} we discuss the obtained result.

\section{Main result\label{sec-main}} 

Recall, that we are given with the domain $\Omega:=\mathfrak{L}(\mathbb{R} \times \omega)\subset\mathbb{R}^3$, where $\mathfrak{L}$ is defined by \eqref{L} and the domain $\omega\subset\mathbb{R}^2$ satisfies \eqref{halfspace}. The rotational angle $\theta(x_1)$ is assumed 
to be  a  continuously differentiable function
 satisfying condition \eqref{3condition}. Additionally, we assume that 
\begin{gather}
\label{1condition}
\dot{\theta}(x_1) \text{ is a monotonically increasing function},
\\\label{2condition}
\dot{\theta}(x_1)\ge0 \text{ on }  \mathbb{R}_+,\quad \dot{\theta}(x_1)\le0 \text{ on } \mathbb{R}_-
\end{gather} 
(for example, one can choose $\theta(x_1)=\sum_{k=0}^m A_k x_1^{2k}$ with $m\in \mathbb{N}$, $A_k\ge 0$, $A_m\not=0$). Another functions also can be treated, see~Subsection~\ref{other}. Note, that 
(\ref{3condition}), (\ref{2condition}) imply
\begin{gather*}
\lim_{|x_1|\to\infty}\theta(x_1)=\infty.
\end{gather*}

We set for $\alpha\in \theta[0, \infty),\,\,\beta\in \theta (-\infty, 0]$:
$$\theta_+^{-1}(\alpha):=\{z\ge0: \theta(z)=\alpha\},\quad \theta_-^{-1}(\beta):=\{z\le0: \theta(z)=\beta\}.$$ 

In what follows 
 for $z\in \mathbb{R}$ we denote 
\begin{gather}\label{pm} 
(z)_\pm:=|z\pm |z||/2 
\end{gather}
(i.e., the negative and positive parts of $z$). \medskip

Our main result is the following theorem.
\begin{theorem} \label{th1}
Let $\sigma\ge 0$.
Under the above assumptions on $\theta$ and $\omega$, for any  $0<\varepsilon<1$ and $\Lambda\ge0$ the following inequality holds true, 
\begin{equation}\label{mainineq}\mathrm{tr}\left(-\Delta_D^\Omega-\Lambda\right)_-^\sigma\le
\frac{L_\sigma}{(1-\varepsilon)^{3/2}} |\omega| \int_{\mathbb{R}}\,\left(\varepsilon f(x_1)-\Lambda\right)_-^{\sigma+3/2}\, \mathrm{d}x_1, 
\end{equation}
where 
\begin{multline}\label{f} 
f(x_1)=\left(\dot{\theta}\left(\theta^{-1}_+ (\theta(x_1)-\pi)\right)\right)^2\,\chi_{\{x_1\ge\theta_+^{-1}(\theta(0)+2\pi)\}}(x_1)  \\ +\left(\dot{\theta}\left(\theta_-^{-1}(\theta(x_1)-\pi)\right)\right)^2\,\chi_{\{x_1\le\theta_-^{-1}(\theta(0)+2\pi)\}}(x_1),
\end{multline} 
and $L_\sigma$ is a constant depending on $\sigma$. For $\sigma\ge 3/2$
\eqref{mainineq} is valid with 
 $L_\sigma=L_{\sigma,3}^{\mathrm{cl}}$, where $L_{\sigma,3}^{\mathrm{cl}}$ is given by (\ref{semiclassical constant}).
\end{theorem}

\begin{remark}\label{finiteness}
It follows easily from the assumptions on the function $\theta(x_1)$  that
$$f(x_1)\to \infty\text{ as }|x_1|\to \infty,$$ which implies the finiteness of the integral  at the right-hand-side of \eqref{mainineq}.
\end{remark}

\section{Proof of Theorem~\ref{th1} \label{sec-proof}}
  
We fix a point $x=(x_2, x_3)\subset\mathbb{R}^2$   and denote  $$\omega_x=\{x_1\in \mathbb{R}: (x_1, x_2, x_3)\in\Omega\}.$$ 
It is easy to see that $\omega_x$ is either the empty set or
a sequence of segments $(a_k (x), b_k (x))_{k=-\infty}^\infty$
satisfying 
\begin{gather*}
a_k(x)<b_k(x)<a_{k+1}(x),\ \forall k\in\mathbb{Z},\\
a_k(x) \to \pm\infty\text{ as }k\to\pm\infty.
\end{gather*}
We assume that these intervals are renumbered in such a way that
\begin{gather*}
b_{-1}(x)<0,\ a_1(x)>0. 
\end{gather*}

In what follows we use the same notation for  $u\in \mathcal{H}_0^1(\Omega)$ and
its extension by zero to the whole $\mathbb{R}^3$ (the resulting function will belong to $\mathcal{H}^1(\mathbb{R}^3)$).

\begin{lemma} \label{lemma-est}
For each $u\in\mathcal{H}_0^1(\Omega)$
\begin{equation}\label{int}\int_{\Omega}\left|\frac{\partial u}{\partial x_1}(x_1, x)\right|^2\,\mathrm{d}x_1\,\mathrm{d}x
\ge\int_{\Omega}f(x_1)|u|^2\,\mathrm{d}x_1\,\mathrm{d}x,
\end{equation}
where $f(x_1)$ is defined by \eqref{f}.
 \end{lemma} 

\begin{proof}
Let us fix $x=(x_2, x_3)\subset\mathbb{R}^2$.
Since  $u(a_k(x), x)=u(b_k(x), x)=0$  one has the following Friedrich inequality for each $k\in\mathbb{Z}$:
$$\int_{a_k(x)}^{b_k(x)}\left|\frac{\partial u}{\partial x_1}\right|^2\,\mathrm{d}x_1\ge\frac{\pi^2}{(b_k(x)-a_k(x))^2}\int_{a_k(x)}^{b_k(x)}|u|^2\,\mathrm{d}x_1,$$ 
whence
\begin{equation}\label{estimate1}\int_{\omega^\pm_x} \left|\frac{\partial u}{\partial x_1}\right|^2\,\mathrm{d}x_1\ge\sum_{k\in\mathbb{Z}:\,\, \pm k\geq   1}
\frac{\pi^2}{(b_k(x)-a_k(x))^2}\int_{a_k(x)}^{b_k(x)}|u|^2\,\mathrm{d}x_1,
\end{equation} 
where $\omega_x^\pm:=\omega_x\cap \mathbb{R}_\pm$.
Our aim is to establish a $\mathit{uniform}$ (with respect to $x$)   estimate from below for the right hand side of (\ref{estimate1}).
We will do this for $\omega^+_x$, for $\omega^-_x$ the arguments are similar.

At first we notice that on the way from $a_k(x)$ to $b_k(x)$ the cross-section $\omega$ turns 
by the angle which is not greater than $\pi$, i.e.
\begin{equation}\label{relation}
\theta(b_k(x))-\theta(a_k(x))\le\pi.
\end{equation}
This follows easily from \eqref{halfspace} and the definition of $a_k$ and $b_k$.
Then, using the mean value theorem and \eqref{3condition}, 
we obtain from \eqref{relation}:
\begin{gather}  
\label{distance}
b_k(x)-a_k(x)\le \frac{\pi}{\mathrm{min}_{s\in[a_k(x), b_k(x)]}\left(\dot{\theta}(s)\right)}\le\frac{\pi}{\dot{\theta}(a_k(x ))},\quad k\ge1.
\end{gather} 
Also from \eqref{relation} we get, using   
the monotonicity of $\theta$ (see \eqref{2condition}),
\begin{gather}\label{ineq1}
a_k(x)\ge\theta_+^{-1}(\theta(b_k(x))-\pi)\ge\theta_+^{-1}(\theta(x_1)-\pi),\quad x_1\in[a_k(x), b_k(x)],\quad k\ge1
\end{gather}
provided 
\begin{gather}\label{pi}
\theta(a_k(x))\ge\theta(0)+\pi.
\end{gather}
Condition \eqref{pi} is required to guarantee 
$$\theta(x_1)-\pi\in \mathrm{dom}(\theta_+^{-1}) =[\theta(0),\infty)\text{\quad as }x_1\in [a_k(x),b_k(x)].$$
Then, again using \eqref{3condition}, we conclude from \eqref{ineq1}:
\begin{equation}
\label{dot}
\dot{\theta}(a_k(x))\ge\dot{\theta}(\theta_+^{-1}(\theta(x_1)-\pi)),\quad x_1\in[a_k(x), b_k(x)],\quad k\ge1.
\end{equation}

Using inequalities (\ref{distance}) and \eqref{dot} we can estimate from below 
the summands in the right-hand side of (\ref{estimate1}) (recall, that now we consider its ``$+$'' part) which correspond to $k$ satisfying \eqref{dot}; the remaining summands we estimate by zero.
As a result we obtain
\begin{equation}\label{intomega+}
\int_{\omega^+_x}\left|\frac{\partial u}{\partial x_1}\right|^2\,\mathrm{d}x_1
\ge\sum_{k:\,\, \theta(a_k(x))\ge\theta(0)+\pi}\,\int_{a_k(x)}^{b_k(x)}\left(\dot{\theta}\left(\theta_+^{-1}(\theta(x_1)-\pi)\right)\right)^2\,
 |u|^2\,\mathrm{d}x_1.
\end{equation} 

We need more information on the location of the smallest 
$a_{k}(x)$ satisfying \eqref{pi}. 
Let $k_0$ be such that 
$\theta(a_{k_0}(x))\geq\theta(0)+\pi$, while 
$\theta(a_{k_0-1}(x))<\theta(0)+\pi$.
There are two possibilities: either

\begin{itemize}
\item all intervals $(a_k(x), b_k(x)),\,\,k<k_0$ belong to $(0, \theta_+^{-1}(\theta(0)+\pi)]$, or 

\item  $(a_k(x), b_k(x))\subset(0, \theta_+^{-1}(\theta(0)+\pi))$, $k\leq k_0-2$,  $a_{k_0-1}(x)<\theta_+^{-1}(\theta(0)+\pi)$, $b_{k_0-1}(x)>\theta_+^{-1}(\theta(0)+\pi)$. 
\end{itemize}
In the first case   $u$ vanishes on $[\theta_+^{-1}(\theta(0)+\pi), a_{k_0}(x)]$ and, therefore 
 one can replace $\displaystyle\sum_{k: \theta(a_k(x))\ge\theta(0)+\pi}\noindent\int_{a_k(x)}^{b_k(x)}$ in \eqref{intomega+} by  $\displaystyle\int_{\theta_+^{-1}(\theta(0)+\pi)}^\infty$. In the second case we use the following observation: 
on the way from $b_k(x)$ to $a_{k+1}(x)$ the cross-section $\omega$ turns 
on the angle which is larger than $\pi$, i.e.
$\theta(a_{k+1}(x))-\theta(b_k(x))\ge\pi.$
Therefore 
$$\theta(a_{k_0}(x))\ge\theta(b_{k_0-1}(x))+\pi\ge\theta(0)+2\pi,$$
while in view of (\ref{relation})
$$\theta(b_{k_0-1}(x))\le\theta(a_{k_0-1}(x))+\pi\le\theta(0)+2\pi.$$
Consequently, in the second case the right- hand side of (\ref{intomega+}) is not smaller than
$$\int_{\theta_+^{-1}(\theta(0)+2\pi)}^\infty \left(\dot{\theta}\left(\theta_+^{-1}(\theta(x_1)-\pi)\right)\right)^2 |u|^2\,\mathrm{d}x_1.$$
Summarising our conclusions in these two case we finally arrive at 
\begin{gather}
\label{final+}
\int_{\omega_x^+}\left|\frac{\partial u}{\partial x_1}\right|^2\,\mathrm{d}x_1\ge\int_{\theta_+^{-1}(\theta(0)+2\pi)}^\infty \left(\dot{\theta}\left(\theta_+^{-1}(\theta(x_1)-\pi)\right)\right)^2 |u|^2\,\mathrm{d}x_1,
\end{gather} 
with the function $f$  being defined by (\ref{f}). 

Using the same arguments we get similar estimate for   $\omega_x^-$: 
\begin{gather}
\label{final-}
\int_{\omega_x^-}\left|\frac{\partial u}{\partial x_1}\right|^2\,\mathrm{d}x_1\ge\int_{-\infty}^{\theta_-^{-1}(\theta(0)+2\pi)}\left(\dot{\theta}\left(\theta_-^{-1}(\theta(x_1)-\pi)\right)\right)^2 |u|^2\,\mathrm{d}x_1.
\end{gather} 

Taking into account that $\int_{\Omega} g(x_1,x_2,x_3) \mathrm{d}x_1 \mathrm{d}x_2  \mathrm{d}x_2 =\int_{\mathbb{R}^2}\left[\int_{\omega_x}g(x_1,x_2,x_3) \mathrm{d}x_1  \right]\mathrm{d} x$
for each $g\in L^1(\mathbb{R}^3)$ with $\mathrm{supp}(g)\subset \overline{\Omega}$, we get \eqref{int} from \eqref{final+}-\eqref{final-} and the definition of the function $f$. The lemma is proved.
\end{proof}
\medskip 
 
We come back to the proof the theorem.
Let us  fix  $0<\varepsilon<1$ and $\Lambda\ge0$. Given a function  $u\in \mathcal{H}_0^1(\Omega)$ the quadratic form of the Dirichlet Laplacian $-\Delta_\Omega^D$ can be represented as follows,
\begin{multline*}
\int_\Omega |\nabla u|^2\, \mathrm{d}x_1\,\mathrm{d}x_2\,\mathrm{d}x_3= \\=\varepsilon \int_{\Omega}\left|\frac{\partial u}{\partial x_1}\right|^2\,\mathrm{d}x_1\,\mathrm{d}x+\int_\Omega\left((1-\varepsilon)\left|\frac{\partial u}{\partial x_1}\right|^2+\left|\frac{\partial u}{\partial x_2}\right|^2+\left|\frac{\partial u}{\partial x_3}\right|^2\right)\, \mathrm{d}x_1\,\mathrm{d}x.
\end{multline*}
This together with (\ref{int}) yields  
\begin{multline}
\label{Lieb-Thirring}
\int_\Omega\left( |\nabla u|^2-\Lambda |u|^2\right)\, \mathrm{d}x_1\,\mathrm{d}x\ge\varepsilon\int_\Omega f(x_1)|u|^2\,\mathrm{d}x_1\,\mathrm{d}x\\ +\int_\Omega\left((1-\varepsilon)\left|\frac{\partial u}{\partial x_1}\right|^2+\left|\frac{\partial u}{\partial x_2}\right|^2+\left|\frac{\partial u}{\partial x_3}\right|^2-\Lambda |u|^2\right)\, \mathrm{d}x_1\,\mathrm{d}x\\\ge(1-\varepsilon)\int_\Omega\left(|\nabla u|^2+\frac{1}{1-\varepsilon}(\varepsilon f(x_1)-\Lambda)_-|u|^2\right)\, \mathrm{d}x_1\,\mathrm{d}x.
\end{multline}

We introduce the complement $\widehat{\Omega} :=\mathbb{R}^3 \backslash\overline{\Omega}$ and consider the functions of the form $h=u+v$ with $u\in \mathcal{H}_0^1(\Omega)$ and $v\in \mathcal{H}_0^1(\widehat{\Omega})$ which we may regard as functions in $\mathbb{R}^3$ extending them by zero to $\widehat{\Omega}$ and $\Omega$, respectively. Next we extend by zero the potential $\frac{1}{1-\varepsilon}(\varepsilon f(x_1)-\Lambda)_-$ to potential $V$ defined on the whole space $\mathbb{R}^3$. Then (\ref{Lieb-Thirring}) implies
\begin{multline}\label{bracketing}
\int_\Omega \left(|\nabla u|^2-\Lambda |u|^2\right)\, \mathrm{d}x_1\,\mathrm{d}x+\int_{\widehat{\Omega}} |\nabla v|^2\, \mathrm{d}x_1\,\mathrm{d}x\\\ge(1-\varepsilon)\int_{\mathbb{R}^3}(|\nabla h|^2+V |h|^2)\, \mathrm{d}x_1\,\mathrm{d}x.
\end{multline} 

The left-hand side in \eqref{bracketing} is the quadratic form corresponding to the operator $\left(-\Delta_D^\Omega-\Lambda\right) \oplus \left(-\Delta^{\widehat{ \Omega}}_D\right)$, while the right-hand one is the form associated with the operator $(1-\varepsilon)\left(-\Delta+V\right)$  on $L^2(\mathbb{R}^3)$.
Since  $-\Delta^{ \widehat{\Omega}}_D$ is a positive operator, 
we conclude from \eqref{bracketing}, using the minimax principle, that for each  $\sigma\geq 0$
\begin{equation}\label{operatorbound}
\mathrm{tr}\,\left(-\Delta_D^\Omega-\Lambda\right)_-^\sigma \le \,(1-\varepsilon)^\sigma\mathrm{tr}\left(-\Delta+V\right)_-^\sigma.
\end{equation}

Further we apply the
Lieb-Thirring inequality for the operator $-\Delta+V$. Recall, that it reads as 
\begin{gather}\label{LT}
\mathrm{tr}\,\left(-\Delta+V\right)_-^\sigma\le  L_{\sigma}\,\int_{\mathbb{R}^3}V_-^{\sigma+3/2}\,\mathrm{d}x_1\,\mathrm{d}x_2\,\mathrm{d}x_3,\ \sigma\geq 0
\end{gather} 
(for $\sigma=0$ \eqref{LT} is known as Cwikel-Lieb-Rozenblum inequality).
It was proved in \cite{LT76} for $\sigma>0$, and in \cite{R72} for $\sigma=0$.
Moreover, for $\sigma\ge 3/2$ the best constant $L_{\sigma}$ for which \eqref{LT} holds coincides
with $L_{\sigma,3}^{\rm cl}$ given by (\ref{semiclassical constant}) \cite{L97}. 

Using \eqref{LT} we 
obtain from \eqref{operatorbound}: 
\begin{multline}
\mathrm{tr}\,\left(-\Delta_D^\Omega-\Lambda\right)_-^\sigma\le(1-\varepsilon)^\sigma L_{\sigma}\,\int_{\mathbb{R}^3}V_-^{\sigma+3/2}\,\mathrm{d}x_1\,\mathrm{d}x\\ =\frac{L_{\sigma}}{(1-\varepsilon)^{3/2}} \int_{\Omega}\,\left(\varepsilon f(x_1)-\Lambda\right)_-^{\sigma+3/2}\, \mathrm{d}x_1\,\mathrm{d}x,\\=\frac{L_{\sigma}}{(1-\varepsilon)^{3/2}} \int_{\mathbb{R}}\,\int_{\omega(x_1)}\,\left(\varepsilon f(x_1)-\Lambda\right)_-^{\sigma+3/2}\, \mathrm{d}x_1\,\mathrm{d}x\\\label{eq2.7}=\frac{L_{\sigma}}{(1-\varepsilon)^{3/2}} \int_{\mathbb{R}}\,\left(\varepsilon f(x_1)-\Lambda\right)_-^{\sigma+3/2}\,|\omega(x_1)|\, \mathrm{d}x_1,
\end{multline}
where $\omega(x_1)$ is the image of $\omega$ after the rotation.
Since for every $x\in\mathbb{R}$, $|\omega(x_1)|=|\omega|$, inequality (\ref{eq2.7}) immediately implies \eqref{mainineq}.

Theorem~\ref{th1} is proved.

\section{Discussion\label{discussion}}

\subsection{Other choices of $\theta$\label{other}} 

Theorem~\ref{th1} (with accordingly modified function $f$) remains valid for some other choices of $\theta$.

Assume, for example, that $\theta(x_1)$ is a  continuously differentiable function 
satisfying \eqref{3condition} and additionally 
\begin{gather}\label{cond1}
\dot{\theta}(x_1)\ge 0,
\\\label{cond3}
\dot{\theta}(x_1) \text{ is increasing on } \mathbb{R}_+,\quad \dot{\theta}(x_1) \text{ is decreasing on } \mathbb{R}_-
\end{gather}(for example, one can choose $\theta(x_1)=\sum_{k=0}^m A_k x_1^{2k+1}$ with $m\in \mathbb{N}$, $A_k\ge 0$, $A_m\not=0$). Then Theorem~\ref{th1} remains valid with $f(x_1)$ being replaced by 
\begin{multline}\label{tildef}
 \widetilde{f}(x_1)=\left(\dot{\theta}\left(\theta^{-1}_+(\theta(x_1)-\pi)\right)\right)^2\,\chi_{\{x_1\ge\theta_+^{-1}(\theta(0)+2\pi)\}}(x_1) \\+\left(\dot{\theta}\left(\theta_-^{-1}(\theta(x_1)+\pi)\right)\right)^2\,\chi_{\{x_1\le\theta_-^{-1}(\theta(0)-2\pi)\}}(x_1).
\end{multline} 

Moreover, if $\theta$ satisfies (\ref{3condition}), \eqref{1condition}, (\ref{2condition}) (respectively,  (\ref{3condition}), (\ref{cond1}), (\ref{cond3})) only for $|x_1|\ge s_0 >0$ then Theorem~\ref{th1}  holds with $f(x_1)$ \eqref{f} being replaced by 
$$f(x_1) \chi_{\mathrm{max}\{s_0, \theta_+^{-1}(2\pi+\theta(s_0))\}}(x_1)+f(x_1)\chi_{\mathrm{min}\{-s_0, \theta_-^{-1}(2\pi+\theta(-s_0))\}}(x_1)$$ 
(respectively, with $\widetilde f(x_1)$ \eqref{tildef} being replaced by 
$$ \widetilde f(x_1) \chi_{\mathrm{max}\{s_0, \theta_+^{-1}(2\pi+\theta(s_0))\}}(x_1)+\widetilde f(x_1)\chi_{\mathrm{min}\{-s_0, \theta_-^{-1}(-2\pi+\theta(-s_0))\}}(x_1)\, ).$$

The proof for the above cases is almost the same as in the case (\ref{3condition}), \eqref{1condition}, (\ref{2condition}).

\subsection{Asymptotics of the obtained bound for large $\Lambda$\label{sec-asymp}}
 
The right-hand side  of  (\ref{mainineq}) looks rather cumbersome.  The situation becomes easier when $\Lambda\to\infty$. The following statement takes place.
  
\begin{proposition}
\label{th2}
The right-hand side of \eqref{mainineq} has the  asymptotics
\begin{equation}\label{lastresult}
(1+ {o}(1)) \frac{L_{\sigma}}{(1-\varepsilon)^{3/2}} |\omega| \int_{\mathbb{R}}(\varepsilon \dot{\theta}^2(x_1)-\Lambda)_-^{\sigma+3/2}\,\mathrm{d}x_1 \text{\, as }\Lambda\to\infty.
\end{equation}
\end{proposition}

\begin{proof} Let us prove that the right hand side of (\ref{mainineq}) can be estimated from above by
the  expression of the form \eqref{lastresult}.

Due to the monotonicity of $\theta$ (see \eqref{2condition}) and $\dot{\theta}$ (see \eqref{1condition}) there exists   $s_0>0$ such that $\theta(x_1)>\theta(0)+\pi$ and
$\dot{\theta}(\theta(x_1)-\pi)\geq\alpha>0$ as $x_1>s_0$.
On each finite interval being contained in $[\theta_+^{-1}(\theta(0)+\pi),\infty$) the function $\theta_+^{-1}(\theta(x_1)-\pi)-x_1$ is bounded.
Moreover, applying the mean value theorem for $x_1>s_0$ one gets
$$\theta_+^{-1}(\theta(x_1)-\pi)-x_1
=\theta_+^{-1}(\theta(x_1)-\pi)-\theta_+^{-1}(\theta(x_1))=-\frac{\pi}{\dot{\theta}(c(x_1))},$$
where $c(x_1)\in(\theta(x_1)-\pi, \theta(x_1))$. Hence
$$K_1:=\mathrm{sup}_{\{x_1\ge\theta_+^{-1}(\theta(0)+\pi)\}}\left|\theta_+^{-1}(\theta(x_1)-\pi)-x_1\right|<\infty$$
Similarly, 
$$
K_2:=\mathrm{sup}_{\{x_1\le\theta_-^{-1}(\theta(0)+\pi)\}}\left|\theta_-^{-1}(\theta(x_1)-\pi)-x_1\right|<\infty.
$$

Let $K=\mathrm{max}\{K_1, K_2\}$. Then 
\begin{multline}
\label{long-est}
\int_{\mathbb{R}}(\varepsilon f(x_1)-\Lambda)_-^{\sigma+3/2}\,\mathrm{d}x_1
=\int_{\theta_+^{-1}(\theta(0)+2\pi)}^\infty(\varepsilon \dot{\theta}^2(\theta_+^{-1}(\theta(x_1)-\pi))-\Lambda)_-^{\sigma+3/2}\,\mathrm{d}x_1\\+\int_{-\infty}^{\theta_-^{-1}(\theta(0)+2\pi)}(\varepsilon \dot{\theta}^2(\theta_-^{-1}(\theta(x_1)-\pi))-\Lambda)_-^{\sigma+3/2}\,\mathrm{d}x_1\\\le\int_{\theta_+^{-1}(\theta(0)+\pi)}^\infty(\varepsilon \dot{\theta}^2(x_1-K)-\Lambda)_-^{\sigma+3/2}\,\mathrm{d}x_1+\int_{-\infty}^{\theta_-^{-1}(\theta(0)+2\pi)}(\varepsilon \dot{\theta}^2(x_1+K)-\Lambda)_-^{\sigma+3/2}\,\mathrm{d}x_1\\=\int_{\theta_+^{-1}(\theta(0)+2\pi)-K}^\infty(\varepsilon \dot{\theta}^2(x_1)-\Lambda)_-^{\sigma+3/2}\,\mathrm{d}x_1+\int_{-\infty}^{\theta_-^{-1}(\theta(0)+2\pi)+K}(\varepsilon \dot{\theta}^2(x_1)-\Lambda)_-^{\sigma+3/2}\,\mathrm{d}x_1\\\leq\int_{\mathbb{R}}(\varepsilon \dot{\theta}^2(x_1)-\Lambda)_-^{\sigma+3/2}\,\mathrm{d}x_1+\Lambda^{\sigma+3/2} \left(|\theta_+^{-1}(\theta(0)+2\pi)-K|+|\theta_-^{-1}(\theta(0)+2\pi)+K|\right).
\end{multline}

One has:
\begin{multline}\label{measure}\int_{\mathbb{R}}(\varepsilon \dot{\theta}^2(x_1)-\Lambda)_-^{\sigma+3/2}\,\mathrm{d}x_1=\int_{|\dot{\theta}(x_1)|\le\sqrt{\Lambda/( \varepsilon)}}(\varepsilon \dot{\theta}^2(x_1)-\Lambda)^{\sigma+3/2}\,\mathrm{d}x_1\\\ge
\int_{|\dot{\theta}(x_1)|\le\sqrt{\Lambda/(2\varepsilon)}}(\varepsilon \dot{\theta}^2(x_1)-\Lambda)^{\sigma+3/2}\,\mathrm{d}x_1 \ge\frac{\Lambda^{\sigma+3/2}}{2^{\sigma+3/2}}\mathrm{meas}
\left\{x_1\in\mathbb{R}:\ |\dot{\theta}(x_1)|\le\sqrt{\Lambda/(2\varepsilon)}\right\}.
\end{multline}
Evidently, the measure standing at the right-hand side of \eqref{measure} tends to
infinity as $\Lambda\to 0$. Therefore \eqref{measure} implies
\begin{gather}\label{o}
\Lambda^{\sigma+3/2} =o \left(\int_{\mathbb{R}}(\varepsilon \dot{\theta}^2(x_1)-\Lambda)_-^{\sigma+3/2}\,\mathrm{d}x_1\right) ,\quad\Lambda\to\infty.
\end{gather} 

Combining \eqref{long-est} and \eqref{o} we get the desired estimate
$$
\int_{\mathbb{R}}(\varepsilon f(x_1)-\Lambda)_-^{\sigma+3/2}\,\mathrm{d}x_1\le (1+ {o}(1))\int_{\mathbb{R}}(\varepsilon \dot{\theta}^2(x_1)-\Lambda)_-^{\sigma+3/2}\,\mathrm{d}x_1.$$

Using the same arguments one can prove that
the right hand side of (\ref{mainineq}) can be also estimated  \textit{from below } 
by the expression of the form \eqref{lastresult}.
The proof is similar: the chain of estimates \eqref{mainineq} remains valid
if we replace all ``$\leq$'' by ``$\geq$'', all 	``$\pm K$'' by ``$\mp K$'' and 
 ``$+\Lambda^{\sigma+3/2}$'' in the last line by	``$-\Lambda^{\sigma+3/2}$''. 

Proposition~\ref{th2} is proved. 
\end{proof}

\subsection{Comparison with the classical Berezin bound}

In this subsection we show that the obtained 
estimate \eqref{mainineq} can be used to improve the classical
Berezin bound for  \emph{bounded} twisted tubes with sufficiently
large rotation velocity
 in the regime
$\Lambda\ll N$, where $N$ is the length of the tube.

Let $\Omega$ be a twisted tube considered in Section~\ref{sec-main}.
Additionally, we assume that its rotation velocity satisfies
\begin{gather} \label{theta-dot}
 |\dot\theta(x_1)| \geq  |x_1|.
\end{gather}
Combining (\ref{mainineq}) and \eqref{lastresult} and taking into account \eqref{theta-dot} one gets for large $\Lambda$: 
\begin{multline} \label{berezin comparison}
\mathrm{tr}\left(-\Delta_D^{\Omega }-\Lambda\right)_-^\sigma  \le(1+\overline{o}(1))\frac{L_\sigma}{(1-\varepsilon)^{3/2}} |\omega|\int_{\mathbb{R}}(\varepsilon\dot{\theta}^2(x_1)-\Lambda)^{\sigma+3/2}\,\mathrm{d}x_1\\ \le (1+\overline{o}(1))\frac{L_\sigma}{(1-\varepsilon)^{3/2}} |\omega|\Lambda^{\sigma+3/2}(\dot{\theta}_+^{-1}(\sqrt{\Lambda/\varepsilon})-\dot{\theta}_-^{-1}(\sqrt{\Lambda/\varepsilon}))\\\le 2 (1+\overline{o}(1))\frac{L_\sigma}{(1-\varepsilon)^{3/2} \sqrt{\varepsilon}} |\omega|\Lambda^{\sigma+2}.
\end{multline} 

Now, we consider the bounded twisted tube
$$\Omega_N:=\left\{(x_1,x_2,x_3)\in\Omega:\ 0<x_1<N\right\}.$$ 
For this tube the classical Berezin inequality \eqref{classical-berezin} reads
\begin{equation}\label{classical}
\mathrm{tr}\left(-\Delta_D^{\Omega_1}-\Lambda\right)_-^\sigma\le L_{\sigma, 3}^{\mathrm{cl}}\Lambda^{\sigma+3/2}|\Omega_N|=
L_{\sigma, 3}^{\mathrm{cl}}\Lambda^{\sigma+3/2}|\omega|N.
\end{equation} 
On the other hand, applying the Dirichlet bracketing technique, we get
$$(-\Delta_D^{\Omega_N})\oplus(-\Delta_D^{\Omega\setminus\overline{\Omega_N}})
\ge -\Delta_D^\Omega,$$
whence
\begin{gather*}
\mathrm{tr}\left(-\Delta_D^{\Omega_N}-\Lambda\right)_-^\sigma
\leq
\mathrm{tr}\left((-\Delta_D^{\Omega_N})\oplus(-\Delta_D^{\Omega\setminus\overline{\Omega_N}})-\Lambda\right)_-^\sigma 
\leq
\mathrm{tr}(-\Delta_D^\Omega-\Lambda)_-^\sigma.
\end{gather*} 
Thus the right-hand side of \eqref{berezin comparison} is also 
an upper bound  $\mathrm{tr}\left(-\Delta_D^{\Omega_N}-\Lambda\right)_-^\sigma$.

Finally, assume that $N=N(\lambda)$ and $\Lambda\ll N$ as $\Lambda\to \infty$. Then for large $\Lambda$
the right-hand side of \eqref{berezin comparison} gives much better extimate for  $\mathrm{tr}\left(-\Delta_D^{\Omega_N}-\Lambda\right)_-^\sigma$ than 
the  classical Berezin inequality \eqref{classical}.

\section*{Acknowledgments}

The work of D.B. is supported by the Czech Science Foundation (GA\v{C}R) within the project 17-01706S and the Czech-Austrian Grant CZ 02/2017 7AMBL7ATO22. 
The authors are grateful to anonymous referee for careful reading of the manuscript and for attracting our attention to new open questions. D.B. thanks D.Krej\v{c}i\v{r}\'{i}k for useful discussions.
A.K. is supported by the Austrian Science Fund
(FWF) under Project No. M 2310-N32.


\begin{thebibliography}{10}

\bibitem{B67} 
{\sc F.A.~ Berezin},  
{\it Covariant and contravariant symbols of operators,}
Mathematics of the USSR. Izvestiya  {\bf 6} (1973), 1117-1151; 
translation from  Izv. Ross. Akad. Nauk Ser. Mat. 6, 1134-1167 (1972).

\bibitem{BHK15} {\sc Ph.~ Briet, H.~ Hammedi, D.~Krej\v{c}i\v{r}\'{i}k},
{\it Hardy inequalities in globally twisted waveguides,}
Letters in Mathematical Physics {\bf 105}, 7 (2015), 939--958. 

\bibitem{BKRS09} {\sc Ph.~Briet, H.~ Kova\v{r}ik, G.~Raikov, E.~Soccorzi}, 
{\it Eigenvalue asymptotics in a twisted waveguide}, 
Communications in Partial Differential Equations {\bf 34}, 8 (2009), 818--836.

\bibitem{BMT07}
{\sc G.~Bouchitt\'e, M.L.~Mascarenhas, L.~Trabucho,}
{\it On the curvature and torsion effects in one-dimensional waveguides,}
ESAIM: Control, Optimization and Calculus of Variations {\bf 13}, 4 (2007), 793--808. 


\bibitem{CDFK05} 
{\sc B.~ Chenaud, P.~Duclos,  P.~ Freitas, D.~Krej\v{c}i\v{r}\'{i}k}, 
{\it Geometrically induced discrete spectrum in curved tubes}, 
Differential Geometry and its Applications {\bf 23}, 2 (2005), 95--105.

\bibitem{DE95} 
{\sc P.~ Duclos,  P.~Exner}, 
{\it Curvature-induced bound states in quantum waveguides in two and three dimensions}, 
Reviews in Mathematical Physics {\bf 7}, 1 (1995), 73--102. 

\bibitem{EE87}
{\sc D.E.~Edmunds, W.D.~Evans},
{\it Spectral theory and differential operators, }
Oxford University Press, Oxford, 1987. 

\bibitem{EKK08} 
{\sc T.~ Ekholm, H.~ Kova\v{r}\'{i}k,  D.~Krej\v{c}i\v{r}\'{i}k},
{\it A Hardy inequality in twisted waveguides,}
Archive for Rational Mechanics and Analysis {\bf 188}, 2 (2008),  245--264. 


\bibitem{EB14} 
{\sc P.~ Exner, D.~ Barseghyan},
{\it Spectral estimates for Dirichlet Laplacians on perturbed twisted tubes,}
Operators and Matrices {\bf 8}, 1 (2014), 167--183.

\bibitem{EK05} 
{\sc P.~ Exner,  H.~ Kova\v{r}\'{i}k}, 
{\it Spectrum of the Schr{\"o}dinger operator in a perturbed periodically twisted tube,}
 Letters in Mathematical Physics {\bf 73}, 3 (2005), 182--192.

\bibitem{EK15} 
{\sc P.~ Exner,  H.~ Kova\v{r}\'{i}k},
{\it Quantum Waveguides, }
Springer, Heidelberg, 2015. 


\bibitem{ES89} 
{\sc P.~ Exner,  P.~ Seba, }
{\it Bound states in curved quantum waveguides, }
Journal of Mathematical Physics {\bf 30}, 11 (1989), 2574--2580.

\bibitem{GJ92} 
{\sc J.~ Goldstone,  R. L.~ Jaffe, }
{\it Bound states in twisting tubes, }
Physical Review B {\bf 45}, 24 (1992), 14100--14107.


\bibitem{Gr04} 
{\sc V.V.~Grushin, }
{\it On the eigenvalues of a finitely perturbed Laplace operator in infinite cylindrical domains, }
Mathematical Notes {\bf 75}, 3-4 (2004), 331--340; translation from Mat. Zametki 75, No. 3, 360-371 (2004).

\bibitem{Gr05}
{\sc V.V.~Grushin, }
{\it Asymptotic behavior of the eigenvalues of the Schr\"{o}dinger operator with transverse potential in weakly curved infinite cylinders, }
Mathematical Notes {\bf 77}, 4 (2005), 606--613; translation from Mat. Zametki 77, No. 5, 656-664 (2005). 


\bibitem{K15}
{\sc D.~Krej\v{c}i\v{r}\'{i}k, }
{\it Waveguides with asymptotically diverging twisting, }
Applied Mathematics Letters {\bf 46} (2015), 7--10.

\bibitem{K08}
{\sc D.~ Krej\v{c}i\v{r}\'{i}k, }
{\it Twisting versus bending in quantum waveguides,}
Analysis on Graphs and Applications (Cambridge 2007),
Proceedings of Symposia in Pure Mathematics 77, American Mathematical Society, Providence, RI, 2008,   617--636. 

\bibitem{KZ10} 
{\sc D.~ Krej\v{c}i\v{r}\'{i}k,  E.~ Zuazua, }
{\it The Hardy inequality and the heat equation in twisted tubes,}
Journal de Math\'{e}matiques Pures et Appliqu\'{e}es {\bf 4}, 3 (2010), 277--303.

\bibitem{L97}  
{\sc A. ~Laptev, T. ~Weidl, }
{\it Sharp Lieb-Thirring inequalities in high dimensions, }
Acta Mathematica {\bf 184}, 1 (2000), 87--100.

\bibitem{LT76} 
{\sc E.H.~ Lieb, W. ~Thirring, }
{\it Inequalities for the moments of the eigenvalues of the Schr\"odinger Hamiltonian and their relation to Sobolev inequalities}, 
in Studies in Mathematical Physics, Essays in Honor of Valentine Bargmann (E. Lieb, B. Simon and A.S. Wightman, eds.); Princeton Univ. Press, Princeton, 1976; pp. 269--330.


\bibitem{RB95}
{\sc W.~Renger, W.~Bulla,}
{\it Existence of bound states in quantum waveguides
under weak conditions, }
Letters in Mathematical Physics {\bf 35}, q (1995), 1--12.


\bibitem{R72} 
{\sc G.~ Rozenblum, }
{\it Distribution of the discrete spectrum of singular differential operators,} 
Soviet Mathematics. Doklady {\bf 13}  (1972) (1972), 245--249; translation from Dokl. Akad. Nauk SSSR 202, 1012-1015 (1972).



\end{thebibliography}
\end{document}